\documentclass[a4paper,12pt]				{amsart}
\usepackage[T1]								{fontenc} 							
\usepackage[utf8]							{inputenc} 							
\usepackage									{ellipsis} 							
\usepackage									{lmodern} 						
\usepackage									{amssymb, amsmath, amsthm}			
\usepackage									{bm, bbm}							
\usepackage[fixamsmath]						{mathtools} 						
\usepackage[shortlabels]					{enumitem} 						
\usepackage[spacing,kerning]				{microtype}
\usepackage									{geometry}
\usepackage 								{hyperref}
\usepackage[square,numbers]{natbib}
\usepackage									{csquotes}						
\usepackage									{todonotes,color}
\usepackage{subcaption}
\usepackage[foot]{amsaddr}

\allowdisplaybreaks
\microtypecontext{spacing=nonfrench}

\theoremstyle{plain}
\newtheorem{theorem}					{Theorem}[section]
\newtheorem{lemma}			[theorem]	{Lemma}

\newtheorem{proposition}	[theorem]	{Proposition}

\theoremstyle{definition}

\theoremstyle{remark}

\newtheorem*{exa*}						{Example}
\newtheorem*{rem*}						{Remark}

	\DeclareMathOperator	{\IR}		{\mathbb{R}}
	\DeclareMathOperator	{\IE}		{\mathbb{E}} 
\DeclareMathOperator	{\IP}		{\mathbb{P}}

\DeclarePairedDelimiter	\abs		{\lvert}	{\rvert}
\DeclarePairedDelimiter	\norm		{\lVert}	{\rVert}

\newcommand				{\eins}		{\text{$\mathbbm{1}$}}

\renewcommand 			{\epsilon}	{\varepsilon}
\renewcommand			{\phi}		{\varphi}
\renewcommand			{\tilde}	{\widetilde}

\numberwithin			{equation}{section}
\title{Some notes on concentration for $\alpha$-subexponential random variables}
\author		{Holger Sambale}
\address	{Faculty of Mathematics, Bielefeld University, Bielefeld, Germany}
\email{hsambale@math.uni-bielefeld.de}

\begin{document}
	
	\begin{abstract}
	We prove extensions of classical concentration inequalities for random variables which have $\alpha$-subexponential tail decay for any $\alpha \in (0,2]$. This includes Hanson--Wright type and convex concentration inequalities in various situations. In particular, we show uniform Hanson--Wright inequalities and convex concentration results for simple random tensors in the spirit of recent work by Klochkov--Zhivotovskiy \cite{KZ18} and Vershynin \cite{Ver19}.
	\end{abstract}
	\subjclass{Primary 60E15, 60F10, Secondary 46E30, 46N30}
	\keywords{Concentration of measure phenomenon, Orlicz norms, subexponential random variables, Hanson-Wright inequality, convex concentration}
	\date{\today}
	\maketitle
	
	\section{Introduction}
	
The aim of this note is to compile a number of smaller results which extend some classical as well as more recent concentration inequalities for bounded or sub-Gaussian random variables to random variables with heavier (but still exponential-type) tails. In detail, we shall consider random variables $X$ which satisfy
\begin{equation}\label{SubExpT}
\mathbb{P}(|X-\mathbb{E}X| \ge t) \le 2 \exp(-t^\alpha/C_{1,\alpha}^\alpha)
\end{equation}
for any $t \ge 0$, some $\alpha \in (0,2]$ and a suitable constant $C_{1,\alpha} > 0$. Such random variables are sometimes called $\alpha$-subexponential (for $\alpha = 2$, they are subgaussian) or sub-Weibull$(\alpha)$ (cf.\ \cite[Definition 2.2]{KC18}).

There are several equivalent reformulations of \eqref{SubExpT}, e.\,g.\ in terms of $L^p$ norms:
\begin{equation}\label{L^pChar}
    \lVert X \rVert_{L^p} \le C_{2,\alpha} p^{1/\alpha}
\end{equation}
for any $p \ge 1$. Another characterization is that these random variables have finite Orlicz norms of order $\alpha$, 
i.\,e.
\begin{equation}\label{ON}
C_{3,\alpha} \coloneqq \lVert X \rVert_{\Psi_\alpha} \coloneqq \inf \{t > 0 \colon \mathbb{E} \exp ((|X|/t)^\alpha) \le 2 \} < \infty.
\end{equation}
If $\alpha < 1$, $\lVert \cdot \rVert_{\Psi_\alpha}$ is actually a quasi-norm, however many norm-like properties (like a triangle-type inequality) can nevertheless be recovered up to $\alpha$-dependent constants (see e.\,g.\ \cite[Appendix A]{GSS19}). 
In fact, $C_{1,\alpha}$, $C_{2,\alpha}$ and $C_{3,\alpha}$ can be chosen such that they only differ by a constant $\alpha$-dependent factor.

Note that $\alpha$-subexponential random variables have log-convex (if $\alpha \le 1$) or log-concave (if $\alpha \ge 1$) tails, i.\,e.\ $t \mapsto - \log \mathbb{P}(\abs{X} \ge t)$ is convex or concave, respectively. For log-convex or log-concave measures, two-sided $L^p$ norm estimates for polynomial chaos (and as a consequence, concentration bounds) have been established over the last 25 years. In the log-convex case, results of this type have been derived for linear forms in \cite{HMO97} and for forms of any order in \cite{KL15,GSS19}. For log-concave measures, starting with linear forms again in \cite{GK95}, important contributions have been made in \cite{La96,La99,LL03,AL12}.

In this note, we mainly present four different results for functions of $\alpha$-subexponential random variables: a Hanson--Wright type inequality in Section \ref{Section:HWI}, a version of the convex concentration inequality in Section \ref{Section:convconc}, a uniform Hanson--Wright inequality in Section \ref{Section:unifHW} and finally a convex concentration inequality for simple random tensors in Section \ref{Section:AuxL}. These results are partly based on and generalize recent research, e.\,g.\ \cite{KZ18} and \cite{Ver19}. In fact, they partially build upon each other: for instance, in the proofs of Section \ref{Section:AuxL} we apply results both from Section \ref{Section:HWI} and Section \ref{Section:convconc}. A more detailed discussion is provided in each of the sections.

Finally, let us introduce some conventions which we will use in this paper.

\textbf{Notations.} If $X_1, \ldots, X_n$ is a sequence of random variables, we denote by $X = (X_1, \ldots, X_n)$ the corresponding random vector.
Moreover, we shall need the following types of norms throughout the paper:
\begin{itemize}
    \item the norms $\lVert x \rVert_p \coloneqq (\sum_{i=1}^n\abs{x_i}^p)^{1/p}$ for $x \in \mathbb{R}^n$,
    \item the $L^p$ norms $\lVert X \rVert_{L^p} := (\mathbb{E}|X|^p)^{1/p}$ for random variables $X$ (cf.\ \eqref{L^pChar}),
    \item the Orlicz (quasi-) norms $\lVert X \rVert_{\Psi_\alpha}$ as introduced in \eqref{ON},
    \item the Hilbert--Schmidt and operator norms $\lVert A \rVert_\mathrm{HS} \coloneqq (\sum_{i,j} a_{ij}^2)^{1/2}$, $\lVert A \rVert_\mathrm{op} \coloneqq \sup\{ \lVert Ax \rVert_2 \colon \lVert x \rVert_2 = 1\}$ for matrices $A = (a_{ij})$.
\end{itemize}
The constants appearing in this paper (typically denoted $C$ or $c$) may vary from line to line. Without subscript they are assumed to be absolute, if they depend on $\alpha$ (only) we shall write $C_\alpha$ or $c_\alpha$.

\textbf{Acknowledgements.}
This work was supported by the German Research Foundation (DFG) via CRC 1283 ``\textit{Taming uncertainty and profiting from randomness and low regularity in analysis, stochastics and their applications}''. The author would moreover like to thank Arthur Sinulis for carefully reading this paper and many fruitful discussions and suggestions.

\section{A generalized Hanson--Wright inequality}\label{Section:HWI}
Arguably, the most famous concentration result for quadratic form is the Hanson--Wright inequality, which first appeared in \cite{HW71}. We may state it as follows: assuming $X_1, \ldots, X_n$ are centered, independent random variables satisfying $\norm{X_i}_{\Psi_2} \le K$ for any $i$, and $A = (a_{ij})$ is a symmetric matrix, we have for any $t \ge 0$
	\[
	\IP\big(\abs{X^TAX - \IE X^TAX} \ge t\big) \le 2 \exp\Big( - \frac{1}{C} \min\Big( \frac{t^2}{K^4 \norm{A}_{\mathrm{HS}}^2}, \frac{t}{K^2 \norm{A}_{\mathrm{op}}} \Big)\Big).
	\]
For a modern proof, see \cite{RV13}, and for various developments, cf.\ \cite{HKZ12,VW15,Ad15,ALM18}.

In this note, we provide an extension of the Hanson--Wright inequality to random variables with bounded Orlicz norms of any order $\alpha \in (0,2]$. This complements the results in \cite{GSS19}, where the case of $\alpha \in (0,1]$ was considered, while for $\alpha = 2$, we get back the actual Hanson--Wright inequality.

\begin{theorem}\label{HWalpha}
	For any $\alpha \in (0,2]$, let $X_1, \ldots, X_n$ be independent, centered random variables such that $\norm{X_i}_{\Psi_\alpha} \le K$ for any $i$, and $A = (a_{ij})$ be a symmetric matrix. Then, for any $t \ge 0$,
	\[
	\IP\big(\abs{X^TAX - \IE X^TAX} \ge t\big) \le 2 \exp\Big( - \frac{1}{C_\alpha} \min\Big( \frac{t^2}{K^4 \norm{A}_{\mathrm{HS}}^2}, \Big( \frac{t}{K^2 \norm{A}_{\mathrm{op}}}\Big)^{\frac \alpha 2} \Big)\Big).
	\]
\end{theorem}

Theorem \ref{HWalpha} generalizes and implies a number of inequalities for quadratic forms in $\alpha$-subexponential random variables (in particular for $\alpha = 1$) which are spread throughout the literature.
For a detailed discussion, see \cite[Remark 1.7]{GSS19}. Note that it is possible to sharpen the tail estimate given by Theorem \ref{HWalpha}, cf.\ e.\,g.\ \cite[Corollary 1.4]{GSS19} for $\alpha \in (0,1]$ or \cite[Theorem 3.2]{AL12} for $\alpha \in [1,2]$ (in fact, the proof of Theorem \ref{HWalpha} works by evaluating the family of norms used therein). The main benefit of Theorem \ref{HWalpha} is that it uses norms which are easily calculable and in many situations already sufficient for applications. 

Before we give the proof of Theorem \ref{HWalpha}, let us briefly mention that for the standard Hanson--Wright inequality, a number of selected applications can be found in \cite{RV13}. Some of them were generalized to $\alpha$-subexponential random variables with $\alpha \le 1$ in \cite{GSS19}, and it is no problem to extend these proofs to any order $\alpha \in (0,2]$ using Theorem \ref{HWalpha}. Here, we just focus on a single example which yields a concentration result for the Euclidean norm of a linear transformation of a vector $X$ having independent components with bounded Orlicz norms around the Hilbert--Schmidt norm of the transformation matrix. This is a variant and extension of \cite[Proposition 2.1]{GSS19} and will be applied in Section \ref{Section:AuxL}.

\begin{proposition}\label{proposition:EuclideanNormVector}
	Let $X_1, \ldots, X_n$ be independent, centered random variables such that $\IE X_i^2 = 1$ and $\norm{X_i}_{\Psi_\alpha} \le K$ for some $\alpha \in (0,2]$, and let $B \neq 0$ be an $m \times n$ matrix. For any $t \ge 0$ we have
	\begin{equation}\label{ENVr}
    \IP(\abs{\norm{BX}_2 - \norm{B}_{\mathrm{HS}}} \ge t K^2 \norm{B}_{\mathrm{op}}) \le 2\exp(- t^\alpha/C_\alpha ).
	\end{equation}
	In particular, for any $t \ge 0$ it holds
    \begin{equation}\label{nov}
        \IP( \abs{\norm{X}_2 - \sqrt{n}} \ge tK^2) \le 2\exp(- t^\alpha/C_\alpha).
    \end{equation}
\end{proposition}

For the proofs, let us recall some elementary relations which we will use throughout the paper to adjust the constants in the tail bounds we derive.

\textbf{Adjusting constants.}
For any two constants $C_1 > C_2 > 1$ we have for all $r \ge 0$ and $C > 0$
\begin{equation}\label{eqn:constantAdjustment} 
C_1 \exp(-r/C) \le C_2 \exp\Big(-\frac{\log(C_2)}{C\log(C_1)} r\Big)
\end{equation}
whenever the left hand side is smaller or equal to $1$ (cf.\ e.\,g.\ \cite[Eq.\ (3.1)]{SS19}). Moreover, for any $\alpha \in (0,2)$, any $\gamma > 0$ and all $t \ge 0$, we may always estimate
\begin{equation}\label{from2toalpha}
    \exp(-(t/C)^2) \le 2 \exp(-(t/C')^\alpha),
\end{equation}
using $\exp(-s^2) \le \exp(1-s^\alpha)$ for any $s > 0$ and \eqref{eqn:constantAdjustment}. More precisely, we may choose $C' := C/\log^{1/\alpha}(2)$. Note that strictly speaking, the range of $t/C \le 1$ is not covered by \eqref{eqn:constantAdjustment}, however in this case (in particular, choosing $C'$ as suggested) both sides of \eqref{from2toalpha} are at least $1$ anyway so that the right hand side still provides a valid upper bound for any probability.

Let us now turn to the proof of Theorem \ref{HWalpha}. In what follows, we actually show that for any $p \ge 2$,
	\begin{equation}\label{wast}
	\norm{X^TAX - \IE X^TAX}_{L^p} \le C_\alpha K^2 \big(p^{1/2} \norm{A}_{\mathrm{HS}} + p^{2/\alpha} \norm{A}_{\mathrm{op}} \big).
	\end{equation}
From here, Theorem \ref{HWalpha} follows by standard means (cf.\ \cite[Proof of Theorem 3.6]{SS18}). Moreover, we may restrict ourselves to $\alpha \in (1,2]$, since the case of $\alpha \in (0,1]$ has been proven in \cite{GSS19}.

\begin{proof}[Proof of Theorem \ref{HWalpha}]
	First we shall treat the off-diagonal part of the quadratic form. Let $w^{(1)}_i, w^{(2)}_i$ be independent (of each other as well as of the $X_i$) symmetrized Weibull random variables with scale $1$ and shape $\alpha$, i.\,e.\ $w^{(j)}_i$ are symmetric random variables with $\mathbb{P}(\abs{w^{(j)}_i} \ge t) = \exp(-t^\alpha)$. In particular, the $w^{(j)}_i$ have logarithmically concave tails.
	
	Using standard decoupling and symmetrization arguments (cf.\ \cite[Theorem 3.1.1 \& Lemma 1.2.6]{PG99}) as well as \cite[Theorem 3.2]{AL12} in the second inequality,
	for any $p \ge 2$ it holds
	\begin{equation} \label{eqn:LpInequality}
	\norm{\sum_{i \neq j} a_{ij} X_i X_j}_{L^p} \le C_\alpha K^2 \norm{\sum_{i \neq j} a_{ij} w^{(1)}_i w^{(2)}_j}_{L^p} \le C_\alpha K^2(\norm{A}_{\{1,2\},p}^{\mathcal{N}} + \norm{A}_{\{\{1\},\{2\}\}, p}^{\mathcal{N}}),
	\end{equation}
	where the norms $\norm{A}_{\mathcal{J},p}^{\mathcal{N}}$ are defined as in \cite{AL12}. Instead of repeating the general definitions, we will only focus on the case we need in our situation. Indeed, for the symmetric Weibull distribution with parameter $\alpha$ we have (again, in the notation of \cite{AL12}) $N(t) = t^\alpha$, and so for $\alpha \in (1,2]$, it follows that $\hat{N}(t) = \min(t^2, \abs{t}^\alpha)$. Hence, the norms can be written as follows:
	\begin{align*}
	\norm{A}_{\{1,2\},p}^{\mathcal{N}} &= 2 \sup\big\lbrace \sum_{i,j} a_{ij} x_{ij} : \sum_{i = 1}^n \min\big(\sum_{j} x_{ij}^2, \big(\sum_{j} x_{ij}^2\big)^{\alpha/2}\big) \le p \big\rbrace,\\
	\norm{A}_{\{\{1\},\{2\}\}, p}^{\mathcal{N}} &= \sup\big\lbrace \sum_{i,j} a_{ij} x_i y_j : \sum_{i = 1}^n \min(x_i^2, \abs{x_i}^{\alpha}) \le p, \sum_{j = 1}^n \min(y_j^2, \abs{y_j}^{\alpha}) \le p \big\rbrace.
	\end{align*}
	Before continuing with the proof, we next introduce a lemma which will help to rewrite the norms in a more tractable form.
\end{proof}

\begin{lemma}
	For any $p \ge 2$ define 
	\begin{align*}
	I_1(p) &\coloneqq \big\lbrace x  = (x_{ij}) \in \IR^{n \times n} : \sum_{i = 1}^n \min\big( \big( \sum_{j = 1}^n x_{ij}^2\big)^{\alpha/2}, \sum_{j = 1}^n x_{ij}^2 \big) \le p \big\rbrace,\\
	I_2(p) &\coloneqq \big\lbrace x_{ij} = z_i y_{ij} \in \IR^{n \times n} : \sum_{i = 1}^n \min(\abs{z_i}^\alpha, z_i^2) \le p, \max_{i = 1,\ldots,n} \sum_{j = 1}^n y_{ij}^2 \le 1 \big\rbrace.
	\end{align*}
	Then $I_1(p) = I_2(p)$.
\end{lemma}

\begin{proof}
	The inclusion $I_1(p) \supseteq I_2(p)$ is an easy calculation, and the inclusion $I_1(p) \subseteq I_2(p)$ follows by defining $z_i = \norm{(x_{ij})_j}$ and $y_{ij} = x_{ij} / \norm{(x_{ij})_j}$ (or $0$, if the norm is zero).
\end{proof}

\begin{proof}[Proof of Theorem \ref{HWalpha}, continued]
	For brevity, for any matrix $A = (a_{ij})$ let us write $\norm{A}_m \coloneqq \max_{i = 1,\ldots,n} ( \sum_{j = 1}^n a_{ij}^2 )^{1/2}$. Note that clearly, $\norm{A}_m \le \norm{A}_{\mathrm{op}}$.
	
	Now, fix some vector $z \in \IR^n$ such that $\sum_{i = 1}^n \min(\abs{z_i}^\alpha, z_i^2) \le p$. The condition also implies
	\[
	p \ge \sum_{i = 1}^n \abs{z_i}^\alpha \eins_{\{\abs{z_i} > 1\}} + \sum_{i= 1}^n z_i^2 \eins_{\{\abs{z_i} \le 1\}} \ge \max\Big( \sum_{i=1}^n z_i^2 \eins_{\{\abs{z_i} \le 1\}}, \sum_{i = 1}^n \abs{z_i} \eins_{\{\abs{z_i} > 1\}} \Big),
	\]
	where in the second step we used $\alpha \in [1,2]$ to estimate $\abs{z_i}^\alpha \eins_{\{\abs{z_i} > 1\}} \ge \abs{z_i} \eins_{\{\abs{z_i} > 1\}}$. So, given any $z$ and $y$ satisfying the conditions of $I_2(p)$, we can write
	\begin{align*}
	\abs{\sum_{i,j} a_{ij} z_i y_{ij}} &\le \sum_{i = 1}^n \abs{z_i} \big( \sum_{j = 1}^n a_{ij}^2 \big)^{1/2} \big( \sum_{j = 1}^n y_{ij}^2 \big)^{1/2} \le \sum_{i = 1}^n \abs{z_i} \big( \sum_{j = 1}^n a_{ij}^2 \big)^{1/2} \\
	&\le \sum_{i = 1}^n \abs{z_i} \eins_{\{\abs{z_i} \le 1\}} \big( \sum_{j = 1}^n a_{ij}^2 \big)^{1/2} + \sum_{i = 1}^n \abs{z_i} \eins_{\{\abs{z_i} > 1\}} \big( \sum_{j =1 }^n a_{ij}^2 \big)^{1/2} \\
	&\le \norm{A}_{\mathrm{HS}} \big(\sum_{i = 1}^n z_i^2 \eins_{\{\abs{z_i} \le 1\}}\big)^{1/2} + \norm{A}_m \sum_{i = 1}^n \abs{z_i} \eins_{\{\abs{z_i} > 1\}}.
	\end{align*} So, this yields
	\begin{equation} \label{eqn:norm1}
	\norm{A}_{\{1,2\},p}^{\mathcal{N}} \le 2p^{1/2} \norm{A}_{\mathrm{HS}} + 2p \norm{A}_m \le 2p^{1/2} \norm{A}_{\mathrm{HS}} + 2p \norm{A}_{\mathrm{op}}.
	\end{equation}
	
	As for $\norm{A}_{\{\{1\},\{2\}\}, p}^{\mathcal{N}}$, we can use the decomposition $z = z_1 + z_2$, where $(z_1)_i = z_i \eins_{\{\abs{z_i} > 1\}}$ and $z_2 = z - z_1$, and obtain
	\begin{align*}
	\norm{A}_{\{\{1\},\{2\}\}, p}^{\mathcal{N}} &\le \sup\big\lbrace \sum_{ij} a_{ij} (x_1)_i (y_1)_j : \norm{x_1}_\alpha \le p^{1/\alpha}, \norm{y_1}_\alpha \le p^{1/\alpha} \big\rbrace \\ &+ 2 \sup\big\lbrace \sum_{ij} a_{ij} (x_1)_i (y_2)_j : \norm{x_1}_\alpha \le p^{1/\alpha}, \norm{y_2}_2 \le p^{1/2} \big\rbrace \\
	&+ \sup\big\lbrace\sum_{ij} a_{ij} (x_2)_i (y_2)_j : \norm{x_2}_2 \le p^{1/2}, \norm{y_2}_2 \le p^{1/2} \big\rbrace \\
	&= p^{2/\alpha} \sup\lbrace \ldots \rbrace + 2p^{1/\alpha + 1/2} \sup\{\ldots\} + p \norm{A}_{\mathrm{op}}
	\end{align*}
	(in the braces, the conditions $\norm{\cdot}_\beta \le p^{1/\beta}$ have been replaced by $\norm{\cdot}_\beta \le 1$). Clearly, since $\norm{x_1}_\alpha \le 1$ implies $\norm{x_1}_2 \le 1$ (and the same for $y_1$), all of the norms can be upper bounded by $\norm{A}_{\mathrm{op}}$, i.\,e.\ we have
	\begin{equation} \label{eqn:norm2}
	\norm{A}_{\{\{1\},\{2\}\}, p}^{\mathcal{N}} \le (p^{2/\alpha} + 2p^{1/\alpha + 1/2} + p) \norm{A}_{\mathrm{op}} \le 4p^{2/\alpha} \norm{A}_{\mathrm{op}},
	\end{equation}
	where the last inequality follows from $p \ge 2$ and $1/2 \le 1/\alpha \le 1 \le (\alpha+2)/(2\alpha) \le 2/\alpha$.
	
	Combining the estimates \eqref{eqn:LpInequality}, \eqref{eqn:norm1} and \eqref{eqn:norm2} yields
	\begin{align*}
	\norm{\sum_{i,j} a_{ij} X_i X_j}_{L^p} \le C_\alpha K^2\big( 2p^{1/2} \norm{A}_{\mathrm{HS}} + 6p^{2/\alpha} \norm{A}_{\mathrm{op}} \big).
	\end{align*}
	
	To treat the diagonal terms, we use Corollary 6.1 in \cite{GSS19}, as $X_i^2$ are independent and satisfy $\norm{X_i^2}_{\Psi_{\alpha/2}} \le K^2$, so that it yields
	\[
	\IP\big( \abs{\sum_{i = 1}^n a_{ii} (X_i^2 - \IE X_i^2)} \ge t\big) \le 2 \exp\Big( - \frac{1}{C_\alpha K^2} \min\Big( \frac{t^2}{\sum_{i =1}^n a_{ii}^2}, \Big( \frac{t}{\max_{i = 1,\ldots,n} \abs{a_{ii}}}\Big)^{\alpha/2} \Big)\Big).
	\]
	Now it is clear that $\max_{i = 1,\ldots, n} \abs{a_{ii}} \le \norm{A}_{\mathrm{op}}$ and $\sum_{i = 1}^n a_{ii}^2 \le \norm{A}_{\mathrm{HS}}^2$. In particular,
	\[
	\norm{\sum_{i = 1}^n a_{ii} (X_i^2 - \IE X_i^2)}_{L^p} \le C_\alpha K^2(p^{1/2} \norm{A}_{\mathrm{HS}} + p^{2/\alpha} \norm{A}_{\mathrm{op}}).
	\]
	The claim \eqref{wast} now follows from Minkowski's inequality.
\end{proof}

Finally, we prove Proposition \ref{proposition:EuclideanNormVector}.

\begin{proof}[Proof of Proposition \ref{proposition:EuclideanNormVector}] It suffices to prove \eqref{ENVr} for matrices satisfying $\norm{B}_{\mathrm{HS}} = 1$, as otherwise we set $\tilde{B} = B \norm{B}_{\mathrm{HS}}^{-1}$ and use the equality
\[
    \{ \abs{\norm{BX}_2 - \norm{B}_{\mathrm{HS}}} \ge \norm{B}_{\mathrm{op}} t \} = \{ \abs{\norm{\tilde{B}X}_2 - 1} \ge \norm{\tilde{B}}_{\mathrm{op}} t \}.
\]

Now let us apply Theorem \ref{HWalpha} to the matrix $A \coloneqq B^T B$. An easy calculation shows that $\mathrm{trace}(A) = \mathrm{trace}(B^T B) = \norm{B}_{\mathrm{HS}}^2 = 1$, so that we have for any $t \ge 0$
\begin{align*}
    \IP\big( \abs{\norm{BX}_2 - 1} \ge t \big) &\le \IP\big( \abs{\norm{BX}_2^2 - 1} \ge \max(t, t^2) \big) \\
    &\le 2\exp\Big( - \frac{1}{C_\alpha} \min\Big( \frac{\max(t,t^2)^2}{K^4\norm{B}_{\mathrm{op}}^2}, \Big( \frac{\max(t,t^2)}{K^4 \norm{B}_{\mathrm{op}}^2} \Big)^{\alpha/2} \Big) \Big) \\
    &\le 2\exp\Big( - \frac{1}{C_\alpha} \min\Big( \frac{t^2}{K^4 \norm{B}_{\mathrm{op}}^2}, \Big( \frac{t^2}{K^4 \norm{B}_{\mathrm{op}}^2} \Big)^{\alpha/2} \Big) \Big) \\
    &\le 2\exp\Big( - \frac{1}{C_\alpha} \Big( \frac{t}{K^2 \norm{B}_{\mathrm{op}}} \Big)^{\alpha} \Big).
\end{align*}
Here, the first step follows from $\abs{z - 1} \le \min( \abs{z^2 -1}, \abs{z^2 - 1}^{1/2})$ for $z \ge 0$, in the second step we have used the estimates $\norm{A}_{\mathrm{HS}}^2 \le \norm{B}_{\mathrm{op}}^2 \norm{B}_{\mathrm{HS}}^2 = \norm{B}_{\mathrm{op}}^2$ and $\norm{A}_{\mathrm{op}} \le \norm{B}_{\mathrm{op}}^2$ and moreover the fact that since $\mathbb{E}X_i^2 = 1$, $K \ge C_\alpha > 0$ (cf.\ e.\,g.\ \cite[Lemma A.2]{GSS19}), while the last step follows from \eqref{from2toalpha} and \eqref{eqn:constantAdjustment}. Setting $t = K^2s\norm{B}_{\mathrm{op}}$ for $s \ge 0$ finishes the proof of \eqref{ENVr}. Finally, \eqref{nov} follows by taking $m=n$ and $B=I$.
\end{proof}

\section{Convex concentration for random variables with bounded Orlicz norms}\label{Section:convconc}

Assume $X_1, \ldots, X_n$ are independent random variables each taking values in some bounded interval $[a,b]$. Then, by convex concentration as established in \cite{Tal88, JS91, Led97}, for every convex $1$-Lipschitz function $f \colon [a,b]^n \to \mathbb{R}$,
\begin{equation}\label{convconc}
    \mathbb{P}(|f(X) - \mathbb{E}f(X)| > t) \le 2 \exp\Big(-\frac{t^2}{2(b-a)^2}\Big)
\end{equation}
for any $t \ge 0$ (see e.\,g.\ \cite[Corollary 3]{Sa00}).

While convex concentration for bounded random variables is by now standard, there is less literature for unbounded random variables. In \cite{Mar18}, a martingal-type approach is used, leading to a result for functionals with stochastically bounded increments. The special case of suprema of unbounded empirical processes was treated in \cite{Ad08,LV13,LV14}. Another branch of research, begun in \cite{Led97} and continued e.\,g.\ in \cite{Sa00,Sa03,GRS15,GRST17,GRSST18,AS19}, is based on functional inequalities (like Poincar\'{e} or log-Sobolev inequalities) restricted to convex functions and weak transport-entropy inequalities. In \cite[Lemma 1.8]{KZ18}, a generalization of \eqref{convconc} for subgaussian random variables ($\alpha =2$) was proven, which we may extend to any order $\alpha \in (0,2]$.

\begin{proposition}\label{prop-alpha}
	Let $X_1, \ldots, X_n$ be independent random variables, $\alpha \in (0,2]$ and $f \colon \mathbb{R}^n \to \mathbb{R}$ convex and $1$-Lipschitz. Then, for any $t \ge 0$,
	\[
	\mathbb{P}(|f(X) - \mathbb{E}f(X)| > t) \le 2 \exp\Big(-\frac{t^\alpha}{C_\alpha\lVert \max_i |X_i| \rVert_{\Psi_\alpha}^\alpha}\Big).
	\]
	In particular,
	\begin{equation}\label{orliczformulation}
	\lVert f(X) - \mathbb{E}f(X) \rVert_{\Psi_\alpha} \le C_\alpha \lVert \max_i |X_i| \rVert_{\Psi_\alpha}.
	\end{equation}
\end{proposition}

Note that the main results of the following two sections can be regarded as applications of Proposition \ref{prop-alpha}. If $f$ is separately convex only (i.\,e.\ convex is every coordinate with the other coordinates being fixed), it is still possible to prove a corresponding result for the upper tails. Indeed, it is no problem to modify the proof below accordingly, replacing \eqref{convconc} by \cite[Theorem 6.10]{BLM13}. Moreover, note that $\lVert \max_i |X_i| \rVert_{\Psi_\alpha}$ cannot be replaced by $\max_i \lVert |X_i| \rVert_{\Psi_\alpha}$ (a counterexample for $\alpha = 2$ is provided in \cite{KZ18}). In general, the Orlicz norm of $\max_i\abs{X_i}$ will be of order $(\log n)^{1/\alpha}$ (cf.\ Lemma \ref{maxOrl}).

\begin{proof}[Proof of Proposition \ref{prop-alpha}]
    Following the lines of the proof of \cite[Lemma 3.5]{KZ18}, the key step is a suitable truncation which goes back to \cite{Ad08}. Indeed, write
	\begin{equation}\label{trunc}
	X_i = X_i1_{\{|X_i|\le M\}} + X_i1_{\{|X_i| > M\}} \eqqcolon Y_i + Z_i
	\end{equation}
	with $M \coloneqq 8 \mathbb{E}\max_i |X_i|$ (in particular, $M \le C_\alpha \lVert \max_i |X_i| \rVert_{\Psi_\alpha}$, cf.\ \cite[Lemma A.2]{GSS19}),
	and let $Y = (Y_1, \ldots, Y_n)$, $Z = (Z_1, \ldots, Z_n)$. By the Lipschitz property of $f$,
	\begin{align}\label{dreieck}
	\begin{split}
	&\mathbb{P}(|f(X) - \mathbb{E}f(X)| > t)\\
	\le \ &\mathbb{P}(|f(Y) - \mathbb{E}f(Y)| + |f(X) - f(Y)| + |\mathbb{E}f(Y) - \mathbb{E}f(X)| > t)\\
	\le \ &\mathbb{P}(|f(Y) - \mathbb{E}f(Y)| + \lVert Z \rVert_2 + \mathbb{E}\lVert Z \rVert_2 > t),
	\end{split}
	\end{align}
	and hence it suffices to bound the terms in the last line.
	
	Applying \eqref{convconc} to $Y$ and using \eqref{from2toalpha} and \eqref{eqn:constantAdjustment}, we obtain
	\begin{align}\label{tailsA}
	\mathbb{P}(|f(Y) - \mathbb{E}f(Y)| > t) 
	\le 2 \exp\Big(- \frac{t^\alpha}{C_\alpha^\alpha\lVert \max_i |X_i| \rVert_{\Psi_\alpha}^\alpha}\Big).
	\end{align}
	Furthermore, below we will show that
	\begin{equation}\label{OrliczNormEst}
	\lVert \lVert Z \rVert_2 \rVert_{\Psi_\alpha} \le C_\alpha \lVert \max_i |X_i| \rVert_{\Psi_\alpha}.
	\end{equation}
	Hence, for any $t \ge 0$,
	\begin{equation}\label{tailsB}
	\mathbb{P}(\lVert Z \rVert_2 \ge t) \le 2 \exp \Big(- \frac{t^\alpha}{C_\alpha^\alpha\lVert \max_i |X_i| \rVert_{\Psi_\alpha}^\alpha}\Big),
	\end{equation}
	and by \cite[Lemma A.2]{GSS19},
	\begin{equation}\label{tailsC}
	\begin{split}
	\mathbb{E}\lVert Z \rVert_2 \le C_\alpha \lVert \max_i |X_i| \rVert_{\Psi_\alpha}.
	\end{split}
	\end{equation}
	
	Temporarily writing $K \coloneqq C_\alpha \lVert \max_i |X_i| \rVert_{\Psi_\alpha}$, where $C_\alpha$ is large enough so that \eqref{tailsA}, \eqref{tailsB} and \eqref{tailsC} hold, \eqref{dreieck} and \eqref{tailsC} yield
	\[
	\mathbb{P}(|f(X) - \mathbb{E}f(X)| > t)
	\le \mathbb{P}(|f(Y) - \mathbb{E}f(Y)| + \lVert Z \rVert_2 > t - K)
	\]
	if $t \ge K$.
	Using subadditivity and invoking \eqref{tailsA} and \eqref{tailsB}, we obtain
	\[
	\mathbb{P}(|f(X) - \mathbb{E}f(X)| > t) \le 4 \exp\Big(-\frac{(t-K)^\alpha}{(2K)^\alpha}\Big) \le 4 \exp\Big(-\frac{t^\alpha}{c_\alpha(2K)^\alpha}\Big),
	\]
	where the last step holds for $t \ge K+\delta$ for some $\delta > 0$. This bound extends trivially to any $t \ge 0$ (if necessary, by a suitable change of constants). Finally, the constant in front of the exponential may be adjusted to 2 by \eqref{eqn:constantAdjustment}, which finishes the proof.
	
	
	It remains to show \eqref{OrliczNormEst}. To this end, recall the Hoffmann--J{\o}rgensen inequality (cf.\ \cite[Theorem 6.8]{LT91}) in the following form: if $W_1, \ldots, W_n$ are independent random variables, $S_k := W_1 + \ldots + W_k$, and $t \ge 0$ is such that $\mathbb{P}(\max_k |S_k| > t) \le 1/8$, then
	\[
	\mathbb{E}\max_k|S_k| \le 3 \IE \max_i|W_i| + 8t.
	\]
	In our case, we set $W_i \coloneqq Z_i^2$, $t=0$, and note that by Chebyshev's inequality,
	\[
	\mathbb{P}(\max_iZ_i^2 > 0) = \mathbb{P}(\max_i|X_i| > M) \le \mathbb{E}\max_i|X_i|/M = 1/8,
	\]
	and consequently, recalling that $S_k = Z_1^2 + \ldots + Z_k^2$,
	\[
	\mathbb{P}(\max_k |S_k| > 0) \le \mathbb{P}(\max_i Z_i^2 > 0) \le 1/8.
	\]
	Thus, together with \cite[Lemma A.2]{GSS19}, we obtain
	\[
	\mathbb{E}\lVert Z \rVert_2^2 \le 3 \mathbb{E} \max_i Z_i^2 \le C_\alpha \lVert \max_i Z_i^2 \rVert_{\Psi_{\alpha/2}}.
	\]
	Now it is easy to see that $\lVert \max_i Z_i^2 \rVert_{\Psi_{\alpha/2}} \le \lVert \max_i |X_i| \rVert_{\Psi_\alpha}^2$, so that altogether we arrive at
	\begin{equation}\label{HJe}
	\mathbb{E}\lVert Z \rVert_2^2 \le C_\alpha \lVert \max_i |X_i| \rVert_{\Psi_\alpha}^2.
	\end{equation}
	
	Furthermore, by \cite[Theorem 6.21]{LT91}, if $W_1, \ldots, W_n$ are independent random variables with zero mean and $\alpha \in (0,1]$,
	\[
	\lVert \sum_{i=1}^{n} W_i \rVert_{\Psi_\alpha} \le C_\alpha (\lVert \sum_{i=1}^n W_i \rVert_{L^1} + \lVert \max_i |W_i| \rVert_{\Psi_\alpha}).
	\]
	In our case, we consider $W_i = Z_i^2 - \mathbb{E} Z_i^2$ and $\alpha/2$ (instead of $\alpha$). Together with the previous arguments (in particular \eqref{HJe}) and \cite[Lemma A.3]{GSS19}, this yields
	\begin{align*}
	\lVert \sum_{i=1}^n(Z_i^2 - \mathbb{E}Z_i^2) \rVert_{\Psi_{\alpha/2}} &\le C_\alpha (\mathbb{E}|\lVert Z\rVert_2^2 - \mathbb{E} \lVert Z\rVert_2^2| + \lVert \max_i |Z_i^2 - \mathbb{E} Z_i^2| \rVert_{\Psi_{\alpha/2}})\\
	&\le C_\alpha (\mathbb{E}\lVert Z\rVert_2^2 + \lVert \max_i Z_i^2 \rVert_{\Psi_{\alpha/2}})
	\le C_\alpha \lVert \max_i |X_i| \rVert_{\Psi_\alpha}^2.
	\end{align*}
	Combining this with \cite[Lemma A.3]{GSS19} and \eqref{HJe}, we arrive at \eqref{OrliczNormEst}.
\end{proof}

\section{Uniform tail bounds for first and second order chaos}\label{Section:unifHW}

In this section, we discuss bounds for the tails of the supremum of certain chaos-type classes of functions. Even if we are particularly interested in quadratic forms, i.\,e.\ uniform Hanson--Wright inequalities, let us first consider linear forms.

Let $X_1, \ldots, X_n$ be independent random variables, let $\alpha \in (0,2]$, and let $\{a_{i,t} \colon i = 1, \ldots, n, t \in \mathcal{T}\}$ be a compact set of real numbers, where $\mathcal{T}$ is some index set. Consider $g(X) \coloneqq \sup_{t \in \mathcal{T}} \sum_{i=1}^n a_{i,t}X_i$. Clearly, $g$ is convex and has Lipschitz constant $D \coloneqq \sup_{t \in \mathcal{T}} (\sum_{i=1}^n a_{i,t}^2)^{1/2}$. Therefore, applying Proposition \ref{prop-alpha}, we immediately obtain that for any $t \ge 0$,
\begin{equation}\label{SupLinForm}
\mathbb{P}(\abs{g(X)-\mathbb{E}g(X)} \ge t) \le 2 \exp\Big(-\frac{t^\alpha}{C_\alpha D^\alpha\lVert \max_i |X_i| \rVert_{\Psi_\alpha}^\alpha}\Big).
\end{equation}
For bounded random variables, corresponding tail bounds can be found e.\,g.\ in \cite[Eq.\ (14)]{Ma00b}, and choosing $\alpha=2$ we get back this result up to constants.

Our main aim is to derive a second order analogue of \eqref{SupLinForm}, i.\,e.\ a uniform Hanson--Wright inequality. A pioneering result in this direction (for Rademacher variables) can be found in \cite{Tal96a}. Later results include \cite{Ad15} (which requires the so-called concentration property), \cite{KMR14}, \cite{DE17} and \cite{GSS18b} (certain classes of weakly dependent random variables). In \cite{KZ18}, a uniform Hanson--Wright inequality for subgaussian random variables was proven. We may show a similar result for random variables with bounded Orlicz norms of any order $\alpha \in (0,2]$.

\begin{theorem}\label{unifHW}
Let $X_1, \ldots, X_n$ be independent, centered random variables and $K \coloneqq \norm{\max_i \abs{X_i}}_{\Psi_\alpha}$, where $\alpha \in (0,2]$. Let $\mathcal{A}$ be a compact set of real symmetric $n \times n$ matrices, and let $f(X) \coloneqq \sup_{A \in \mathcal{A}} (X^TAX - \mathbb{E}X^TAX)$. Then, for any $t \ge 0$,
\[
\mathbb{P}(f(X) - \mathbb{E}f(X) \ge t) \le 2\exp\Big(-\frac{1}{C_\alpha K^\alpha}\min\Big(\frac{t^\alpha}{(\mathbb{E}\sup_{A\in\mathcal{A}}\norm{AX}_2)^\alpha}, \frac{t^{\alpha/2}}{\sup_{A\in\mathcal{A}}\norm{A}_{\mathrm{op}}^{\alpha/2}}\Big)\Big).
\]
\end{theorem}

For $\alpha = 2$, this gives back \cite[Theorem 1.1]{KZ18} (up to constants and a different range of $t$). Comparing Theorem \ref{unifHW} to Theorem \ref{HWalpha}, we note that instead of a subgaussian term, we obtain an $\alpha$-subexponential term (which can be trivially transformed into a subgaussian term for $t \le \mathbb{E}\sup_{A\in\mathcal{A}}\norm{AX}_2$, but this does not cover the complete $\alpha$-subexponential regime). Moreover, Theorem \ref{unifHW} only gives a bound for the upper tails. Therefore, if $\mathcal{A}$ just consists of a single matrix, Theorem \ref{HWalpha} is stronger. These differences have technical reasons.

To prove Theorem \ref{unifHW}, we shall follow the basic steps of \cite{KZ18} and modify those where the truncation comes in. Let us first repeat some tools and results. In the sequel, for a random vector $W = (W_1, \ldots, W_n)$, we shall denote
\begin{equation}
    f(W) \coloneqq \sup_{A \in \mathcal{A}} (W^TAW - g(A)),
\end{equation}
where $g \colon \mathbb{R}^{n \times n} \to \mathbb{R}$ is some function. Moreover, if $A$ is any matrix, we denote by $\mathrm{Diag}(A)$ its diagonal part (regarded as a matrix with zero entries on its off-diagonal). The following lemma combines \cite[Lemmas 3.2 \& 3.5]{KZ18}.

\begin{lemma}\label{KZ3.1}
\begin{enumerate}
    \item Assume the vector $W$ has independent components which satisfy $W_i \le K$ a.s. Then, for any $t \ge 1$, we have
    \[
    f(W) - \mathbb{E}f(W) \le C\big(K(\mathbb{E}\sup_{A \in \mathcal{A}}\norm{AW}_2 + \mathbb{E}\sup_{A \in \mathcal{A}} \norm{\mathrm{Diag}(A)W}_2)\sqrt{t} + K^2\sup_{A\in \mathcal{A}}\norm{A}_{\mathrm{op}}t\big)
    \]
    with probability at least $1 - e^{-t}$.
    \item Assuming the vector $W$ has independent (but not necessarily bounded) components with mean zero, we have
    \[
    \mathbb{E}\sup_{A\in\mathcal{A}}\norm{\mathrm{Diag}(A)W}_2 \le C\mathbb{E}\sup_{A\in\mathcal{A}}\norm{AW}_2.
    \]
    \end{enumerate}
\end{lemma}

From now on, let $X$ be the random vector from Theorem \ref{unifHW}, and recall the truncated random vector $Y$ which we introduced in \eqref{trunc} (and the corresponding ``remainder'' $Z$). Then, Lemma \ref{KZ3.1} (1) for $f(Y)$ with $g(A) = \mathbb{E}X^TAX$ yields
\begin{equation}\label{3.11}
    f(Y) - \mathbb{E}f(Y) \le C\big(M(\mathbb{E}\sup_{A \in \mathcal{A}}\norm{AY}_2 + \mathbb{E}\sup_{A\in \mathcal{A}}\norm{\mathrm{Diag}(A)}_2)t^{1/\alpha} + M^2t^{2/\alpha}\sup_{A\in\mathcal{A}}\norm{A}_{\mathrm{op}}\big)
\end{equation}
with probability at least $1 - e^{-t}$ (actually, \eqref{3.11} even holds with $\alpha=2$, but in the sequel we will have to use the weaker version given above anyway). Here we recall that $M \le C_\alpha \lVert \max_i |X_i| \rVert_{\Psi_\alpha}$.

To prove Theorem \ref{unifHW}, it remains to replace the terms involving the truncated random vector $Y$ by the original vector $X$. First, by Proposition \ref{prop-alpha} and since $\sup_{A\in\mathcal{A}} \norm{AX}_2$ is $\sup_{A\in\mathcal{A}}\norm{A}_{\mathrm{op}}$-Lipschitz, we obtain
\begin{equation}\label{3.13}
    \mathbb{P}(\sup_{A\in\mathcal{A}} \norm{AX}_2 > \mathbb{E}\sup_{A\in\mathcal{A}}\norm{AX}_2 + C_\alpha\norm{\max_i \abs{X_i}}_{\Psi_\alpha} \sup_{A\in\mathcal{A}} \norm{A}_{\mathrm{op}}t^{1/\alpha}) \le 2e^{-t}.
\end{equation}
Moreover, by \eqref{tailsC},
\begin{equation}\label{3.14}
    \abs{\mathbb{E}\sup_{A\in\mathcal{A}} \norm{AY}_2 - \mathbb{E}\sup_{A\in\mathcal{A}} \norm{AX}_2} \le C_\alpha\norm{\max_i \abs{X_i}}_{\Psi_\alpha} \sup_{A\in\mathcal{A}} \norm{A}_{\mathrm{op}}.
\end{equation}

Next we estimate the difference between the expectations of $f(X)$ and $f(Y)$.

\begin{lemma}\label{KZ3.6}
    We have
    \[
    \abs{\mathbb{E}f(Y)-\mathbb{E}f(X)} \le C_\alpha\big(\norm{\max_i \abs{X_i}}_{\Psi_\alpha} \mathbb{E}\sup_{A\in\mathcal{A}}\norm{AX}_2 + \norm{\max_i \abs{X_i}}_{\Psi_\alpha}^2 \sup_{A \in \mathcal{A}}\norm{A}_\mathrm{op}\big).
    \]
\end{lemma}

\begin{proof}
First note that
\begin{align*}
    f(X) &= \sup_{A\in\mathcal{A}} (Y^T AY - \mathbb{E}X^TAX + Z^TAX + Z^TAY)\\
    &\le \sup_{A\in\mathcal{A}} (Y^T AY - \mathbb{E}X^TAX) + \sup_{A\in\mathcal{A}}\abs{Z^TAX} + \sup_{A\in\mathcal{A}}\abs{Z^TAY}\\
    &\le f(Y) + \norm{Z}_2\sup_{A\in\mathcal{A}} \norm{AX}_2 + \norm{Z}_2\sup_{A\in\mathcal{A}}\norm{AY}_2.
\end{align*}
The same holds if we reverse the roles of $X$ and $Y$. As a consequence,
\begin{equation}\label{3.9}
    \abs{f(X)-f(Y)} \le \norm{Z}_2\sup_{A\in\mathcal{A}} \norm{AX}_2 + \norm{Z}_2\sup_{A\in\mathcal{A}}\norm{AY}_2
\end{equation}
and thus, taking expectations and applying H\"older's inequality,
\begin{equation}\label{3.15}
    \abs{\mathbb{E}f(X)-\mathbb{E}f(Y)} \le (\mathbb{E}\norm{Z}_2^2)^{1/2}((\mathbb{E}\sup_{A\in\mathcal{A}} \norm{AX}_2^2)^{1/2} + (\mathbb{E}\sup_{A\in\mathcal{A}} \norm{AY}_2^2)^{1/2}).
\end{equation}
We may estimate $(\mathbb{E}\norm{Z}_2^2)^{1/2}$ using \eqref{HJe}. Moreover, by related arguments as in \eqref{tailsC}, from \eqref{3.13} we get that
\[
\mathbb{E}\sup_{A\in\mathcal{A}} \norm{AX}_2^2 \le C_\alpha((\mathbb{E}\sup_{A\in\mathcal{A}} \norm{AX}_2)^2 + \norm{\max_i \abs{X_i}}_{\Psi_\alpha}^2 \sup_{A \in \mathcal{A}}\norm{A}_{\mathrm{op}}^2).
\]
Arguing similarly and using \eqref{3.14}, the same bound also holds for $(\mathbb{E}\sup_{A\in\mathcal{A}} \norm{AY}_2^2)^{1/2}$. Taking roots and plugging everything into \eqref{3.15} completes the proof.
\end{proof}

Finally, we prove the central result of this section.

\begin{proof}[Proof of Theorem \ref{unifHW}]
First, it immediately follows from Lemma \ref{KZ3.6} that
\begin{equation}\label{3.16}
    \mathbb{E}f(Y) \le \mathbb{E}f(X) + C_\alpha\big(\norm{\max_i \abs{X_i}}_{\Psi_\alpha} \mathbb{E}\sup_{A\in\mathcal{A}}\norm{AX}_2 + \norm{\max_i \abs{X_i}}_{\Psi_\alpha}^2 \sup_{A \in \mathcal{A}}\norm{A}_{\mathrm{op}}\big).
\end{equation}
Moreover, by \eqref{3.14} and Lemma \ref{KZ3.1} (2),
\begin{equation}\label{3.17}
    \mathbb{E}\sup_{A \in \mathcal{A}}\norm{AY}_2 + \mathbb{E}\sup_{A\in \mathcal{A}}\norm{\mathrm{Diag}(A)Y}_2 \le C_\alpha(\mathbb{E}\sup_{A \in \mathcal{A}}\norm{AX}_2 + \norm{\max_i \abs{X_i}}_{\Psi_\alpha} \sup_{A \in \mathcal{A}}\norm{A}_{\mathrm{op}}).
\end{equation}
Finally, it follows from \eqref{3.9}, \eqref{3.13} and \eqref{3.14} that
\begin{align*}
   \abs{f(X)-f(Y)} &\le \norm{Z}_2\sup_{A\in\mathcal{A}} \norm{AX}_2 + \norm{Z}_2\sup_{A\in\mathcal{A}}\norm{AY}_2\\
   &\le C_\alpha(\norm{Z}_2\mathbb{E}\sup_{A\in\mathcal{A}} \norm{AX}_2 + \norm{Z}_2\norm{\max_i \abs{X_i}}_{\Psi_\alpha}\sup_{A\in\mathcal{A}}\norm{A}_{\mathrm{op}}t^{1/\alpha})
\end{align*}
with probability at least $1 - 4e^{-t}$ for all $t \ge 1$. Using \eqref{tailsB}, it follows that
\begin{equation}\label{3.17b}
   \abs{f(X)-f(Y)} \le C_\alpha(\norm{\max_i \abs{X_i}}_{\Psi_\alpha}\mathbb{E}\sup_{A\in\mathcal{A}} \norm{AX}_2t^{1/\alpha} + \norm{\max_i \abs{X_i}}_{\Psi_\alpha}^2\sup_{A\in\mathcal{A}}\norm{A}_{\mathrm{op}}t^{2/\alpha})
\end{equation}
with probability at least $1 - 6e^{-t}$ for all $t \ge 1$. Combining \eqref{3.16}, \eqref{3.17} and \eqref{3.17b} and plugging into \eqref{3.11} thus yields that with probability at least $1 - 6e^{-t}$ for all $t \ge 1$,
\begin{align*}
f(X) - \mathbb{E}f(X) &\le C_\alpha(\norm{\max_i \abs{X_i}}_{\Psi_\alpha}\mathbb{E}\sup_{A\in\mathcal{A}} \norm{AX}_2t^{1/\alpha} + \norm{\max_i \abs{X_i}}_{\Psi_\alpha}^2\sup_{A\in\mathcal{A}}\norm{A}_{\mathrm{op}}t^{2/\alpha})\\
&\eqqcolon C_\alpha(at^{1/\alpha} + bt^{2/\alpha}).
\end{align*}
If $u \ge \max(a,b)$, it follows that
\[
\mathbb{P}(f(X) - \mathbb{E}f(X) \ge u) \le 6\exp\Big(-\frac{1}{C_\alpha}\min\Big(\Big(\frac{u}{a}\Big)^\alpha, \Big(\frac{u}{b}\Big)^{\alpha/2}\Big)\Big).
\]
By standard means (a suitable change of constants, using \eqref{eqn:constantAdjustment}), this bound may be extended to any $u \ge 0$ and the constant may be adjusted to $2$.
\end{proof}

\section{Random Tensors}\label{Section:AuxL}

By a \emph{simple random tensor}, we mean a random tensor of the form
\begin{equation}\label{SRT}
X \coloneqq X_1 \otimes \cdots \otimes X_d = (X_{1,i_1} \cdots X_{d,i_d})_{i_1, \ldots, i_d} \in \mathbb{R}^{n^d},
\end{equation}
where all $X_k$ are independent random vectors in $\mathbb{R}^n$ whose coordinates are independent, centered random variables with variance one. Concentration results for random tensors (typically for polynomial-type functions) have been shown in \cite{La06, AW15, GSS19}, for instance.

Recently, in \cite{Ver19} new and interesting concentration bounds for simple random tensors were shown. In comparison to previous work, these inequalities focus on \emph{small} values of $t$, e.\,g.\ a regime where subgaussian tail decay holds. Moreover, in contrast to previous papers, \cite{Ver19} provides constants with optimal dependence on $d$. One of these results is the following convex concentration inequality: assuming that $n$ and $d$ are positive integers, $f \colon \mathbb{R}^{n^d} \to \mathbb{R}$ is convex and 1-Lipschitz and the $X_{ij}$ are bounded a.s., then for any $t \in [0, 2 n^{d/2}]$,
\begin{equation}\label{V1.3}
\mathbb{P}(|f(X) - \mathbb{E}f(X)| > t) \le 2 \exp \Big(- \frac{t^2}{Cdn^{d-1}}\Big),
\end{equation}
where $C > 0$ only depends on the bound of the coordinates. Using Theorem \ref{HWalpha} and Proposition \ref{prop-alpha}, we may extend this result to unbounded random variables as follows:

\begin{theorem}\label{V1.3subg}
	Let $n, d \in \mathbb{N}$ and $f \colon \mathbb{R}^{n^d} \to \mathbb{R}$ be convex and 1-Lipschitz. Consider a simple random tensor $X \coloneqq X_1 \otimes \cdots \otimes X_d$ as in \eqref{SRT}. Fix $\alpha \in [1,2]$, and assume that $\lVert X_{i,j} \rVert_{\Psi_\alpha} \le K$. Then, for any $t \in [0, c_\alpha n^{d/2} (\log n)^{1/\alpha}/K]$,
	\[
    \mathbb{P}(|f(X) - \mathbb{E}f(X)| > t) \le 2 \exp\Big(-\frac{1}{C_\alpha}\Big( \frac{t}{d^{1/2} n^{(d-1)/2} (\log n)^{1/\alpha} K}\Big)^\alpha\Big).
    \]
	On the other hand, if $\alpha \in (0,1)$, then, for any $t \in [0, c_\alpha n^{d/2} (\log n)^{1/\alpha}d^{1/\alpha-1/2}/K]$,
	\[
	\mathbb{P}(|f(X) - \mathbb{E}f(X)| > t) \le 2 \exp\Big(-\frac{1}{C_\alpha}\Big( \frac{t}{d^{1/\alpha}n^{(d-1)/2} (\log n)^{1/\alpha} K}\Big)^\alpha\Big).
	\]
\end{theorem}

The logarithmic factor stems from the Orlicz norm of $\max_i |X_i|$ in Proposition \ref{prop-alpha}. For a slightly sharper version which includes the explicit dependence on these norms (and also gives back \eqref{V1.3} for bounded random variables and $\alpha = 2$), see \eqref{V1.3subgsp} in the proof of Theorem \ref{V1.3subg}. We believe that Theorem \ref{V1.3subg} is non-optimal for $\alpha<1$ as we would expect a bound of the same type as for $\alpha \in [1,2]$. However, a key difference in the proofs is that in the case of $\alpha \ge 1$ we can make use of moment-generating functions. This is clearly not possible if $\alpha < 1$, so that less subtle estimates must be invoked instead.

For the proof of Theorem \ref{V1.3subg}, we first adapt some preliminary steps and compile a number of auxiliary lemmas whose proofs are deferred to the appendix. As a start, we need some additional characterizations of $\alpha$-subexponential random variables via the behavior of the moment-generating functions:

\begin{proposition}\label{proptails}
Let $X$ be a random variable and $\alpha \in (0,2]$. Then, the properties \eqref{SubExpT}, \eqref{L^pChar} and \eqref{ON} are equivalent to
\begin{equation}\label{mgfhochalpha}
    \mathbb{E}\exp(\lambda^\alpha |X|^\alpha) \le \exp(C_{4,\alpha}^\alpha \lambda^\alpha)
\end{equation}
for all $0 \le \lambda \le 1/C_{4,\alpha}$. If $\alpha \in [1,2]$ and $\mathbb{E}X = 0$, then the above properties are moreover equivalent to
\begin{equation}\label{mgf}
    \mathbb{E}\exp(\lambda X) \le \begin{cases}
    \exp(C_{5,\alpha}^2 \lambda^2) & \text{if} \ |\lambda| \le 1/C_{5,\alpha} \\ \exp(C_{5,\alpha}^{\alpha/(\alpha-1)} |\lambda|^{\alpha/(\alpha-1)}) & \text{if} \ |\lambda| \ge 1/C_{5,\alpha} \ \text{and} \ \alpha > 1.
    \end{cases}
\end{equation}
The parameters $C_{i,\alpha}$, $i= 1, \ldots, 5$, can be chosen such that they only differ by constant $\alpha$-dependent factors. In particular, we can take $C_{i,\alpha} = c_{i,\alpha}\lVert X \rVert_{\Psi_\alpha}$.
\end{proposition}


To continue, note that $\lVert X \rVert_2 = \prod_{i=1}^d \lVert X_i \rVert_2$. A key step in the proofs of \cite{Ver19} is a maximal inequality which simultaneously controls the tails of $\prod_{i=1}^k \lVert X_i \rVert_2$, $k = 1, \ldots, d$, where the $X_i$ have independent subgaussian components, i.\,e.\ $\alpha=2$. Generalizing these results to any order $\alpha \in (0,2]$ is not hard. The following preparatory lemma extends \cite[Lemma 3.1]{Ver19}. Note that in the proof (given in the appendix again), we apply Proposition \ref{proposition:EuclideanNormVector}.

\begin{lemma}\label{V3.1}
	Let $X_1, \ldots, X_d \in \mathbb{R}^n$ be independent random vectors with independent, centered coordinates such that $\IE X_{i,j}^2 = 1$ and $\lVert X_{i,j} \rVert_{\Psi_\alpha} \le K$ for some $\alpha \in (0,2]$. Then, for any $t \in [0, 2n^{d/2}]$,
	\[
	\mathbb{P} \Big(\prod_{i=1}^d \lVert X_i \rVert_2 > n^{d/2} + t\Big) \le 2 \exp\Big(-\frac{1}{C_\alpha} \Big(\frac{t}{K^2 d^{1/2}n^{(d-1)/2}}\Big)^\alpha\Big).
	\]
\end{lemma}

To control all $k=1, \ldots, d$ simultaneously, we need a generalized version of the maximal inequality \cite[Lemma 3.2]{Ver19} which we state next.

\begin{lemma}\label{V3.2g}
	Let $X_1, \ldots, X_d \in \mathbb{R}^n$ be independent random vectors with independent, centered coordinates such that $\IE X_{i,j}^2 = 1$ and $\lVert X_{i,j} \rVert_{\Psi_\alpha} \le K$ for some $\alpha \in (0,2]$. Then, for any $u \in [0,2]$,
	\[
	\mathbb{P}\Big(\max_{1 \le k \le d}n^{-k/2}\prod_{i=1}^k \lVert X_i \rVert_2 > 1 + u \Big) \le 2 \exp\Big(-\frac{1}{C_\alpha}\Big(\frac{n^{1/2} u}{K^2d^{1/2}}\Big)^\alpha\Big).
	\]
\end{lemma}

The following martingale-type bound is directly taken from \cite{Ver19}:

\begin{lemma}[\cite{Ver19}, Lemma 4.1]\label{V4.1}
	Let $X_1, \ldots X_d$ be independent random vectors. For each $k = 1, \ldots, d$, let $f_k = f_k(X_k, \ldots, X_d)$ be an integrable real-valued function and $\mathcal{E}_k$ be an event that is uniquely determined by the vectors $X_k, \ldots, X_d$. Let $\mathcal{E}_{d+1}$ be the entire probability space. Suppose that for every $k = 1, \ldots, d$ we have
	\[
	\mathbb{E}_{X_k} \exp(f_k) \le \pi_k
	\]
	for every realisation of $X_{k+1}, \ldots, X_d$ in $\mathcal{E}_{k+1}$. Then, for $\mathcal{E} \coloneqq \mathcal{E}_2 \cap \cdots \cap \mathcal{E}_d$, we have
	\[
	\mathbb{E} \exp(f_1 + \ldots + f_d)1_\mathcal{E} \le \pi_1 \cdots \pi_d.
	\]
\end{lemma}

Finally, we need a bound for the Orlicz norm of $\max_i |X_i|$.

\begin{lemma}\label{maxOrl}
    Let $X_1, \ldots, X_n$ be independent, centered random variables such that $\lVert X_i \rVert_{\Psi_\alpha} \le K$ for any $i$ and some $\alpha > 0$. Then,
    \[
    \lVert \max_i |X_i| \rVert_{\Psi_\alpha} \le C_\alpha K \max\Big\{\Big(\frac{\sqrt{2}+1}{\sqrt{2}-1}\Big)^{1/\alpha},(\log n)^{1/\alpha}\Big(\frac{2}{\log 2}\Big)^{1/\alpha}\Big\}.
    \]
    Here, we may choose $C_\alpha= \max\{2^{1/\alpha-1}, 2^{1-1/\alpha}\}$.
\end{lemma}

Note that for $\alpha \ge 1$, \cite[Proposition 4.3.1]{PG99} provides a similar result. However, we are also interested in the case of $\alpha < 1$ in the present note. The condition $\mathbb{E}X_i = 0$ in Lemma \ref{maxOrl} can easily be removed only at the expense of a different absolute constant.

We are now ready to prove Theorem \ref{V1.3subg}.



\begin{proof}[Proof of Theorem \ref{V1.3subg}]
We shall adapt the arguments from \cite{Ver19}. First let
\[
\mathcal{E}_k \coloneqq \Big\{ \prod_{i=k}^d \lVert X_i \rVert_2 \le 2 n^{(d-k+1)/2}\Big\},\qquad k = 1, \ldots, d,
\]
and let $\mathcal{E}_{d+1}$ be the full space. It then follows from Lemma \ref{V3.2g} for $u=1$ that
\begin{equation}\label{E^cg}
\mathbb{P}(\mathcal{E}) \ge 1 - 2 \exp\Big(-\frac{1}{C_\alpha}\Big(\frac{n^{1/2}}{K^2d^{1/2}}\Big)^\alpha\Big),
\end{equation}
where $\mathcal{E} \coloneqq \mathcal{E}_2 \cap \cdots \cap \mathcal{E}_d$.

Now fix any realization $x_2, \ldots, x_d$ of the random vectors $X_2, \ldots, X_d$ in $\mathcal{E}_2$ and apply Proposition \ref{prop-alpha} to the function $f_1(x_1)$ given by $x_1 \mapsto f(x_1, \ldots x_d)$. Clearly, $f_1$ is convex, and since
\[
|f(x \otimes x_2 \otimes \cdots \otimes x_d) - f(y \otimes x_2 \otimes \cdots \otimes x_d)| \le \lVert x - y \rVert_2 \prod_{i=2}^d \lVert x_i \rVert_2 \le \lVert x - y \rVert_2 2 n^{(d-1)/2},
\]
we see that it is $2 n^{(d-1)/2}$-Lipschitz. Hence, it follows from \eqref{orliczformulation} that
\begin{equation}\label{convexsteps1}
\lVert f - \mathbb{E}_{X_1} f \rVert_{\Psi_\alpha(X_1)} \le c_\alpha n^{(d-1)/2} \lVert \max_j |X_{1,j}| \rVert_{\Psi_\alpha}
\end{equation}
for any $x_2, \ldots, x_d$ in $\mathcal{E}_2$, where $\mathbb{E}_{X_1}$ denotes taking the expectation with respect to $X_1$ (which, by independence, is the same as conditionally on $X_2, \ldots, X_d$).

To continue, fix any realization $x_3, \ldots, x_d$ of the random vectors $X_3, \ldots, X_d$ which satisfy $\mathcal{E}_3$ and apply Proposition \ref{prop-alpha} to the function $f_2(x_2)$ given by $x_2 \mapsto \mathbb{E}_{X_1}f(X_1, x_2, \ldots, x_d)$. Again, $f_2$ is a convex function, and since
\begin{align*}
&|\mathbb{E}_{X_1}f(X_1 \otimes x \otimes x_3 \otimes \ldots \otimes x_d) -  \mathbb{E}_{X_1}f(X_1 \otimes y \otimes x_3 \otimes \ldots \otimes x_d)|\\
&\le \mathbb{E}_{X_1} \lVert X_1 \otimes (x-y) \otimes x_3 \otimes \ldots \otimes x_d \rVert_2 \le (\mathbb{E}\lVert X_1 \rVert_2^2)^{1/2} \lVert x - y \rVert_2 \prod_{i=3}^d \lVert x_i \rVert_2\\
&\le \sqrt{n} \lVert x - y \rVert_2 \cdot 2 n^{(d-2)/2} = \lVert x - y \rVert_2 \cdot 2 n^{(d-1)/2},
\end{align*}
$f_2$ is $2 n^{(d-1)/2}$-Lipschitz. Applying \eqref{orliczformulation}, we thus obtain
\begin{equation}\label{convexsteps2}
\lVert \mathbb{E}_{X_1} f  - \mathbb{E}_{X_1,X_2} f \rVert_{\Psi_\alpha(X_2)} \le c_\alpha n^{(d-1)/2} \lVert \max_j |X_{2,j}| \rVert_{\Psi_\alpha}
\end{equation}
for any $x_3, \ldots, x_d$ in $\mathcal{E}_3$. Iterating this procedure, we arrive at
\begin{equation}\label{convexstepsg}
\lVert \mathbb{E}_{X_1, \ldots, X_{k-1}} f  - \mathbb{E}_{X_1,\ldots, X_k} f \rVert_{\Psi_\alpha(X_k)} \le c_\alpha n^{(d-1)/2} \lVert \max_j |X_{k,j}| \rVert_{\Psi_\alpha}
\end{equation}
for any realization $x_{k+1}, \ldots, x_d$ of $X_{k+1}, \ldots, X_d$ in $\mathcal{E}_{k+1}$.

We now combine \eqref{convexstepsg} for $k = 1, \ldots, d$. To this end, we write 
\[
\Delta_k \coloneqq \Delta_k(X_k, \ldots, X_d) \coloneqq \mathbb{E}_{X_1, \ldots, X_{k-1}}f - \mathbb{E}_{X_1, \ldots, X_k}f,
\]
and apply Proposition \ref{proptails}. Here we have to distinguish between the cases where $\alpha \in [1,2]$ and $\alpha \in (0,1)$. If $\alpha \ge 1$, we use \eqref{mgf} to arrive at a bound for the moment-generating function. Writing $M_k \coloneqq \lVert \max_j |X_{k,j}| \rVert_{\Psi_\alpha}$, we obtain
\[
\mathbb{E}\exp(\lambda \Delta_k) \le \begin{cases}
    \exp((c_\alpha n^{(d-1)/2} M_k)^2 \lambda^2)\\ \exp((c_\alpha n^{(d-1)/2}M_k)^{\alpha/(\alpha-1)} |\lambda|^{\alpha/(\alpha-1)})
    \end{cases}
\]
for all $x_{k+1}, \ldots, x_d$ in $\mathcal{E}_{k+1}$, where the first line holds if $|\lambda| \le 1/(c_\alpha n^{(d-1)/2}M_k)$ and the second one if $|\lambda| \ge 1/(c_\alpha n^{(d-1)/2}M_k)$ and $\alpha > 1$. For the simplicity of presentation, temporarily assume that $c_\alpha n^{(d-1)/2} = 1$ (alternatively, replace $M_k$ by $c_\alpha n^{(d-1)/2}M_k$ in the following arguments) and that $M_1 \le \ldots \le M_d$. Using Lemma \ref{V4.1}, we obtain
\begin{align*}
&\qquad \mathbb{E} \exp(\lambda(f-\mathbb{E}f))1_\mathcal{E} = \mathbb{E}\exp(\lambda(\Delta_1 + \cdots + \Delta_d))1_\mathcal{E}\\
&\le \exp((M_1^2 + \ldots + M_k^2) \lambda^2 + (M_{k+1}^{\alpha/(\alpha-1)} + \ldots + M_d^{\alpha/(\alpha-1)}) |\lambda|^{\alpha/(\alpha-1)})
\end{align*}
for $|\lambda| \in [1/M_{k+1},1/M_k]$, where we formally set $M_0 \coloneqq 0$ and $M_{d+1} \coloneqq \infty$. In particular, setting $M \coloneqq (M_1^2 + \ldots + M_d^2)^{1/2}$, we have
\[
\mathbb{E} \exp(\lambda(f-\mathbb{E}f))1_\mathcal{E} \le \exp(M^2 \lambda^2)
\]
for all $|\lambda| \le 1/M_d = 1/(\max_k M_k)$. Furthermore, for $\alpha > 1$ it is not hard to see that
\[
(M_1^2 + \ldots + M_k^2) \lambda^2 + (M_{k+1}^{\alpha/(\alpha-1)} + \ldots + M_d^{\alpha/(\alpha-1)}) |\lambda|^{\alpha/(\alpha-1)} \le M^{\alpha/(\alpha-1)}|\lambda|^{\alpha/(\alpha-1)}
\]
if $|\lambda| \in [1/M_{k+1},1/M_k]$ for some $k = 0, 1, \ldots, d-1$ or $|\lambda| \in [1/M,1/M_d]$ for $k=d$. Indeed, by monotonicity (divide by $\lambda^2$ and compare the coefficients) it suffices to check this for $\lambda = 1/M_{k+1}$ or $\lambda = 1/M$ if $k=d$. The cases of $k=0$ and $k=d$ follow by simple calculations. In the general case, set $x^2 = (M_1^2 + \ldots + M_{k+1}^2)/M_{k+1}^2$ and $y^{\alpha/(\alpha-1)} = (M_{k+2}^{\alpha/(\alpha-1)} + \ldots + M_d^{\alpha/(\alpha-1)})/M_{k+1}^{\alpha/(\alpha-1)}$. Clearly, $(x^2 + y^{\alpha/(\alpha-1)})^{(\alpha-1)/\alpha} \le (x^2 + y^2)^{1/2}$ since $x \ge 1$ and $\alpha/(\alpha-1) \ge 2$. Moreover, $y^2 \le (M_{k+2}^2 + \ldots + M_d^2)/M_{k+1}^2$, which proves the inequality. Altogether, inserting the factor $c_\alpha n^{(d-1)/2}$ again, we therefore obtain
\begin{align}\label{estimateMGF}
&\mathbb{E} \exp(\lambda(f-\mathbb{E}f))1_\mathcal{E} = \mathbb{E}\exp(\lambda(\Delta_1 + \cdots + \Delta_d))1_\mathcal{E}\notag\\
&\qquad \le \begin{cases}
    \exp((c_\alpha n^{(d-1)/2})^2 M^2 \lambda^2) \\ \exp((c_\alpha n^{(d-1)/2})^{\alpha/(\alpha-1)} M^{\alpha/(\alpha-1)} |\lambda|^{\alpha/(\alpha-1)}),
    \end{cases}
\end{align}
where the first line holds if $|\lambda| \le 1/(c_\alpha n^{(d-1)/2}M)$ and the second one if $|\lambda| \ge 1/(c_\alpha n^{(d-1)/2}M)$ and $\alpha > 1$.

On the other hand, if $\alpha < 1$, we use \eqref{mgfhochalpha}. Together with Lemma \ref{V4.1} and the subadditivity of $|\cdot|^\alpha$ for $\alpha \in (0,1)$, this yields
\begin{align}\label{estimateMGFg}
\begin{split}
&\mathbb{E} \exp(\lambda^\alpha|f-\mathbb{E}f|^\alpha)1_\mathcal{E} \le  \mathbb{E}\exp(\lambda^\alpha(|\Delta_1|^\alpha + \cdots + |\Delta_d|^\alpha))1_\mathcal{E}\\
&\le \exp((c_\alpha n^{(d-1)/2})^\alpha (M_1^\alpha + \cdots + M_d^\alpha) \lambda^\alpha)
\end{split}
\end{align}
for $\lambda \in [0, 1/(c_\alpha n^{(d-1)/2} \max_k M_k)]$.

To finish the proof, first consider $\alpha \in [1,2]$. Then, for any $\lambda > 0$, we have
\begin{align}\label{estimate1}
\begin{split}
\mathbb{P}(f - \mathbb{E}f > t) &\le \mathbb{P}(\{f - \mathbb{E}f > t \} \cap \mathcal{E}) + \mathbb{P}(\mathcal{E}^c)\\
&\le \mathbb{P}(\exp(\lambda(f-\mathbb{E}f))1_\mathcal{E} > \exp(\lambda t)) + \mathbb{P}(\mathcal{E}^c)\\
&\le \exp\Big(- \Big(\frac{t}{c_\alpha n^{(d-1)/2} M}\Big)^\alpha\Big) + 2 \exp\Big(-\frac{1}{C_\alpha}\Big(\frac{n^{1/2}}{K^2d^{1/2}}\Big)^\alpha\Big).
\end{split}
\end{align}
where the last step follows by standard arguments (similarly as in the proof of Proposition \ref{proptails} given in the appendix), using \eqref{estimateMGF} and \eqref{E^cg}.
Now, assume that $t \le c_\alpha n^{d/2} M/(K^2d^{1/2})$. Then, the right-hand side of \eqref{estimate1} is dominated by the first term (possibly after adjusting constants), so that we arrive at
\[
\mathbb{P}(f - \mathbb{E}f > t) \le 3 \exp\Big(- \frac{1}{C_\alpha} \Big(\frac{t}{n^{(d-1)/2} M}\Big)^\alpha\Big).
\]
The same arguments hold if $f$ is replaced by $-f$. Adjusting constants by \eqref{eqn:constantAdjustment}, we obtain that for any $t \in [0, c_\alpha n^{d/2} M/(K^2d^{1/2})]$,
	\begin{equation}\label{V1.3subgsp}
	\mathbb{P}(|f(X) - \mathbb{E}f(X)| > t) \le 2 \exp\Big(-\frac{1}{C_\alpha}\Big( \frac{t}{n^{(d-1)/2} M}\Big)^\alpha\Big).
	\end{equation}
Now it remains to note that by Lemma \ref{maxOrl}, we have
\[
\lVert \max_j |X_{i,j}| \rVert_{\Psi_\alpha} \le C_\alpha (\log n)^{1/\alpha} \max_j \lVert X_{i,j} \rVert_{\Psi_\alpha} \le C_\alpha (\log n)^{1/\alpha} K.
\]

If $\alpha \in (0,1)$, similarly to \eqref{estimate1}, using \eqref{estimateMGFg}, \eqref{E^cg} and Proposition \ref{proptails},
\begin{align*}
\mathbb{P}(|f - \mathbb{E}f| > t) &\le \mathbb{P}(\{|f - \mathbb{E}f| > t \} \cap \mathcal{E}) + \mathbb{P}(\mathcal{E}^c)\\
&\le 2\exp\Big(- \Big(\frac{t}{c_\alpha n^{(d-1)/2} M_\alpha}\Big)^\alpha\Big) + 2 \exp\Big(-\frac{1}{C_\alpha}\Big(\frac{n^{1/2}}{K^2d^{1/2}}\Big)^\alpha\Big),
\end{align*}
where $M_\alpha \coloneqq (M_1^\alpha + \ldots + M_d^\alpha)^{1/\alpha}$. The rest follows as above.
\end{proof}

\appendix
\section{}

\begin{proof}[Proof of Proposition \ref{proptails}]
The equivalence of \eqref{SubExpT}, \eqref{L^pChar}, \eqref{ON} and \eqref{mgfhochalpha} is easily seen by directly adapting the arguments from the proof of \cite[Proposition 2.5.2]{Ver18}. 
To see that these properties imply \eqref{mgf}, first note that since in particular $\lVert X \rVert_{\Psi_1} < \infty$, the bound for $|\lambda| \le 1/C_{5,\alpha}'$ directly follows from \cite{Ver18}, Proposition 2.7.1 (e).
To see the bound for large values of $|\lambda|$, we infer that by the weighted arithmetic-geometric mean inequality (with weights $\alpha-1$ and $1$),
\[y^{(\alpha-1)/\alpha} z^{1/\alpha} \le \frac{\alpha-1}{\alpha} y + \frac{1}{\alpha} z\]
for any $y,z \ge 0$. Setting $y \coloneqq |\lambda|^{\alpha/(\alpha-1)}$ and $z \coloneqq |x|^\alpha$, we may conclude that
\[\lambda x \le \frac{\alpha-1}{\alpha} |\lambda|^{\alpha/(\alpha-1)} + \frac{1}{\alpha} |x|^\alpha\]
for any $\lambda, x \in \mathbb{R}$. Consequently, using \eqref{mgfhochalpha} assuming $C_{4,\alpha} = 1$, for any $|\lambda| \ge 1$
\begin{align*}
    \mathbb{E} \exp(\lambda X) &\le \exp\big(\frac{\alpha-1}{\alpha} |\lambda|^{\alpha/(\alpha-1)}\big)\mathbb{E}\exp(|X|^\alpha/\alpha)\\ &\le \exp\big(\frac{\alpha-1}{\alpha} |\lambda|^{\alpha/(\alpha-1)}\big) \exp(1/\alpha) \le \exp(|\lambda|^{\alpha/(\alpha-1)}).
\end{align*}
This yields \eqref{mgf} for $|\lambda| \ge 1/C_{5,\alpha}''$. The claim now follows by taking $C_{5,\alpha} \coloneqq \max (C_{5,\alpha}', C_{5,\alpha}'')$.

Finally, starting with \eqref{mgf} assuming $C_{5,\alpha}=1$, let us check \eqref{SubExpT}. To this end, note that for any $\lambda > 0$,
\[
    \mathbb{P}(X \ge t) \le \exp(-\lambda t) \mathbb{E}\exp(\lambda X) \le \exp(-\lambda t + \lambda^2 \eins_{\{\lambda \le 1\}} +  \lambda^{\alpha/(\alpha-1)}\eins_{\{\lambda > 1\}}).
\]
Now choose $\lambda \coloneqq t/2$ if $t \le 2$, $\lambda \coloneqq ((\alpha-1)t/\alpha)^{\alpha-1}$ if $t \ge \alpha/(\alpha-1)$ and $\lambda \coloneqq 1$ if $t \in (2, \alpha/(\alpha-1))$. This yields
\[\mathbb{P}(X \ge t) \le \begin{cases}
\exp(-t^2/4) & \text{if} \ t \le 2,\\
\exp(-(t-1)) & \text{if} \ t \in (2, \alpha/(\alpha-1)),\\
\exp(-\frac{(\alpha-1)^{\alpha-1}}{\alpha^\alpha} t^\alpha) & \text{if} \ t \ge \alpha/(\alpha-1).
\end{cases}\]
Now use \eqref{eqn:constantAdjustment}, \eqref{from2toalpha} and the fact that $\exp(-(t-1)) \le \exp(- t^\alpha/C_\alpha^\alpha)$ for any $t \in (2, \alpha/(\alpha-1))$. It follows that
\[\mathbb{P}(X \ge t) \le 2 \exp(- t^\alpha/C_{1,\alpha}'^\alpha)\]
for any $t \ge 0$. The same argument for $-X$ completes the proof.
\end{proof}

\begin{proof}[Proof of Lemma \ref{V3.1}]
	By the arithmetic and geometric means inequality and since $\mathbb{E}\lVert X_i \rVert_2 \le \sqrt{n}$, for any $s \ge 0$,
	\begin{align}\begin{split}\label{st}
	\mathbb{P}\Big(\prod_{i=1}^d \lVert X_i \rVert_2 > (\sqrt{n} + s)^d \Big) &\le \mathbb{P}\Big(\frac{1}{d}\sum_{i = 1}^d (\lVert X_i \rVert_2 - \sqrt{n}) > s\Big)\\
	&\le \mathbb{P}\Big(\frac{1}{d}\sum_{i = 1}^d (\lVert X_i \rVert_2 - \mathbb{E}\lVert X_i \rVert_2) > s\Big).
	\end{split}
	\end{align}
	Moreover, by \eqref{nov} and \cite[Corollary A.5]{GSS19},
	\[
	\big\lVert \lVert X_i \rVert_2 - \mathbb{E}\lVert X_i \rVert_2 \big\rVert_{\Psi_\alpha} = \big\lVert \lVert X_i \rVert_2 - \sqrt{n} - (\mathbb{E}\lVert X_i \rVert_2 - \sqrt{n}) \big\rVert_{\Psi_\alpha} \le C_\alpha K^2
	\]
	for any $i = 1, \ldots, d$. On the other hand, if $Y_1, \ldots, Y_d$ are independent centered random variables with $\lVert Y_i \rVert_{\Psi_\alpha} \le M$, we have
	\begin{align*}
	\mathbb{P}\Big(\frac{1}{d}\Big|\sum_{i=1}^d Y_i\Big| \ge s\Big) &\le 2 \exp\Big(-\frac{1}{C_\alpha}\min\Big( \Big(\frac{s\sqrt{d}}{M}\Big)^2, \Big(\frac{s\sqrt{d}}{M}\Big)^\alpha \Big)\Big)\\
	&\le 2 \exp\Big(-\frac{1}{C_\alpha} \Big(\frac{s\sqrt{d}}{M}\Big)^\alpha \Big).
	\end{align*}
	Here, the first estimate follows from \cite{GK95} ($\alpha > 1$) and \cite{HMO97} ($\alpha \le 1$), while the last step follows by \eqref{from2toalpha}.
	As a consequence, \eqref{st} can be bounded by $2 \exp(-s^\alpha d^{\alpha/2} / (K^{2\alpha}C_\alpha))$.
	
	For $u \in [0,2]$ and $s = u \sqrt{n}/2d$, we have $ (\sqrt{n} + s)^d \le n^{d/2}(1+u)$. Plugging in, we arrive at
	\[
	\mathbb{P}\Big(\prod_{i=1}^d \lVert X_i \rVert_2 > n^{d/2}(1 + u) \Big) \le 2 \exp\Big(-\frac{1}{C_\alpha} \Big(\frac{n^{1/2} u}{K^2 d^{1/2}}\Big)^\alpha\Big).
	\]
	Now set $u  \coloneqq t/n^{d/2}$.
\end{proof}

\begin{proof}[Proof of Lemma \ref{V3.2g}]
	Let us first recall the partition into ``binary sets'' which appears in the proof of \cite[Lemma 3.2]{Ver19}. Here we assume that $d = 2^L$ for some $L \in \mathbb{N}$ (if not, increase $d$). Then, for any $\ell \in \{0,1, \ldots, L \}$, we consider the partition $\mathcal{I}_\ell$ of $\{1, \ldots, d\}$ into $2^\ell$ successive (integer) intervals of length $d_\ell \coloneqq d/2^\ell$ which we call ``binary intervals''. It is not hard to see that for any $k = 1, \ldots, d$, we can partition $[1,k]$ into binary intervals of different lengths such that this partition contains at most one interval of each family $\mathcal{I}_\ell$.
	
	Now it suffices to prove that
	\[
	\mathbb{P} \Big(\exists 0 \le \ell \le L, \exists I \in \mathcal{I}_\ell \colon \prod_{i\in I} \lVert X_i \rVert_2 > (1 + 2 ^{-\ell/4}u)n^{d_\ell/2}\Big) \le 2 \exp\Big(-\frac{1}{C_\alpha} \Big(\frac{n^{1/2}u}{K^2d^{1/2}}\Big)^\alpha \Big)
	\]
	(cf.\ Step 3 of the proof of \cite[Lemma 3.2]{Ver19}, where the reduction to this case is explained in detail). To this end, for any $\ell \in \{0, 1, \ldots, L\}$, any $I \in \mathcal{I}_\ell$ and $d_\ell \coloneqq |I| = d/2^\ell$, we apply Lemma \ref{V3.1} for $d_\ell$ and $t \coloneqq 2^{-\ell/4} n^{d_\ell/2}u$. This yields
	\begin{align*}
	\mathbb{P} \Big(\prod_{i\in I} \lVert X_i \rVert_2 > (1 + 2 ^{-\ell/4}u)n^{d_\ell/2}\Big) &\le 2 \exp\Big(-\frac{1}{C_\alpha} \Big(\frac{n^{1/2}u}{2^{\ell/4}K^2d_\ell^{1/2}}\Big)^\alpha\Big)\\ &= 2 \exp\Big(-\frac{1}{C_\alpha} \Big(2^{\ell/4}\frac{n^{1/2}u}{K^2d^{1/2}}\Big)^\alpha\Big).
	\end{align*}
	Altogether, we arrive at
	\begin{align}\label{ZwSchr}
	\begin{split}
	\mathbb{P} \Big(\exists \ell \in \{0, 1, \ldots, L\}, \exists I \in \mathcal{I}_\ell \colon &\prod_{i\in I} \lVert X_i \rVert_2 > (1 + 2 ^{-\ell/4}u)n^{d_\ell/2}\Big)\\ &\le \sum_{\ell=0}^L 2^\ell \cdot 2 \exp\Big(-\frac{1}{C_\alpha}\Big(2^{\ell/4}\frac{n^{1/2}u}{K^2d^{1/2}}\Big)^\alpha\Big).
	\end{split}
	\end{align}
	We may now assume that $(n^{1/2} u/(K^2d^{1/2}))^\alpha/C_\alpha \ge 1$ (otherwise the bound in Lemma \ref{V3.2g} gets trivial by adjusting $C_\alpha$). Using the elementary inequality $ab \ge (a+b)/2$ for all $a,b \ge 1$, we arrive at
	\[
	2^{\ell\alpha/4} \frac{1}{C_\alpha}\Big(\frac{n^{1/2}u}{K^2d^{1/2}}\Big)^\alpha \ge \frac{1}{2} \Big(2^{\ell\alpha/4} + \frac{1}{C_\alpha} \Big(\frac{n^{1/2}u}{K^2 d^{1/2}}\Big)^\alpha\Big).
	\]
	Using this in \eqref{ZwSchr}, we obtain the upper bound
	\[
	2 \exp\Big(- \frac{1}{2C_\alpha} \Big(\frac{n^{1/2}u}{K^2d^{1/2}}\Big)^\alpha\Big) \sum_{\ell=0}^L 2^\ell \exp(- 2^{\ell\alpha/4 - 1}) \le c_\alpha \exp\Big(-\frac{1}{2C_\alpha} \Big(\frac{n^{1/2}u}{K^2d^{1/2}}\Big)^\alpha \Big).
	\]
	By \eqref{eqn:constantAdjustment}, we can assume $c_\alpha=2$.
\end{proof}

To prove Lemma \ref{maxOrl}, we first present a number of lemmas and auxiliary statements. In particular, recall that if $\alpha \in (0, \infty)$, then for any $x,y \in (0,\infty)$,
    \begin{equation}\label{auxlem1}
    c_\alpha(x^\alpha + y^\alpha) \le (x+y)^\alpha \le \tilde{c}_\alpha(x^\alpha+y^\alpha),
    \end{equation}
where $c_\alpha \coloneqq 2^{\alpha-1} \wedge 1$ and $\tilde{c}_\alpha \coloneqq 2^{\alpha-1} \vee 1$. Indeed, if $\alpha \le 1$, using the concavity of the function $x \mapsto x^\alpha$ it follows by standard arguments that $2^{\alpha-1}(x^\alpha + y^\alpha) \le (x+y)^\alpha \le x^\alpha + y^\alpha$. Likewise, for $\alpha \ge 1$, using the convexity of $x \mapsto x^\alpha$ we obtain $x^\alpha + y^\alpha \le (x+y)^\alpha \le 2^{\alpha-1}(x^\alpha + y^\alpha)$.

\begin{lemma}\label{auxlem2}
    Let $X_1, \ldots, X_n$ be independent, centered random variables such that $\lVert X_i \rVert_{\Psi_\alpha} \le 1$ for some $\alpha > 0$. Then, if $Y \coloneqq \max_i |X_i|$ and $c \coloneqq (c_\alpha^{-1}\log n)^{1/\alpha}$, we have
    \[
    \mathbb{P}(Y \ge c + t) \le 2 \exp(-c_\alpha t^\alpha)
    \]
    with $c_\alpha$ as in \eqref{auxlem1}.
\end{lemma}

\begin{proof}
We have
\begin{align*}
    \mathbb{P}(Y \ge c + t) &\le n \mathbb{P}(|X_i| \ge c + t) \le 2n \exp(-(c+t)^\alpha)\\
    &\le 2n \exp(-c_\alpha (t^\alpha + c^\alpha) = 2 \exp(-c_\alpha t^\alpha),
\end{align*}
where we have used \eqref{auxlem1} in the next-to-last step.
\end{proof}

\begin{lemma}\label{auxlem3}
    Let $Y \ge 0$ be a random variable which satisfies
    \[
    \mathbb{P}(Y \ge c + t) \le 2 \exp(-t^\alpha)
    \]
    for some $c \ge 0$ and any $t \ge 0$. Then,
    \[
    \lVert Y \rVert_{\Psi_\alpha} \le \tilde{c}_\alpha^{1/\alpha} \max\Big\{\Big(\frac{\sqrt{2}+1}{\sqrt{2}-1}\Big)^{1/\alpha},c\Big(\frac{2}{\log 2}\Big)^{1/\alpha}\Big\}
    \]
    with $\tilde{c}_\alpha$ as in \eqref{auxlem1}.
\end{lemma}

\begin{proof}
By \eqref{auxlem1} and monotonicity, we have $Y^\alpha \le \tilde{c}_\alpha((Y-c)_+^\alpha + c^\alpha)$, where $x_+ \coloneqq \max(x,0)$. Thus,
\begin{align*}
\mathbb{E}\exp\Big(\frac{Y^\alpha}{s^\alpha}\Big) &\le \exp\Big(\frac{\tilde{c}_\alpha c^\alpha}{s^\alpha}\Big) \mathbb{E}\exp\Big(\frac{\tilde{c}_\alpha (Y-c)_+^\alpha}{s^\alpha}\Big)\\
&= \exp\Big(\frac{c^\alpha}{t^\alpha}\Big) \mathbb{E}\exp\Big(\frac{(Y-c)_+^\alpha}{t^\alpha}\Big) \eqqcolon I_1 \cdot I_2,
\end{align*}
where we have set $t \coloneqq s\tilde{c}_\alpha^{-1/\alpha}$. Obviously, $I_1 \le \sqrt{2}$ if $t \ge c(1/\log\sqrt{2})^{1/\alpha}$. As for $I_2$, we have
\begin{align*}
    I_2 &= 1 + \int_1^\infty \mathbb{P}((Y-c)_+\ge t(\log y)^{1/\alpha}) dy\\
    &\le 1 + 2 \int_1^\infty \exp(-t^\alpha \log y) dy = 1 + 2\int_1^\infty \frac{1}{y^{t^\alpha}} dy \le \sqrt{2}
\end{align*}
if $t \ge ((\sqrt{2}+1)/(\sqrt{2}-1))^{1/\alpha}$. Therefore, $I_1I_2 \le 2$ if $t \ge \max\{((\sqrt{2}+1)/(\sqrt{2}-1))^{1/\alpha}, c(2/\log 2)^{1/\alpha}\}$, which finishes the proof.
\end{proof}

Having these lemmas at hand, the proof of Lemma \ref{maxOrl} is easily completed.

\begin{proof}[Proof of Lemma \ref{maxOrl}]
The random variables $\hat{X}_i \coloneqq X_i/K$ obviously satisfy the assumptions of Lemma \ref{auxlem2}. Hence, setting $Y \coloneqq \max_i |\hat{X}_i| = K^{-1} \max_i |X_i|$,
\[
\mathbb{P}(c_\alpha^{1/\alpha}Y \ge (\log n)^{1/\alpha} + t) \le 2 \exp(-t^\alpha).
\]
Therefore, we may apply Lemma \ref{auxlem3} to $\hat{Y} \coloneqq c_\alpha^{1/\alpha} K^{-1} \max_i |X_i|$. This yields
    \[
    \lVert \hat{Y} \rVert_{\Psi_\alpha} \le \tilde{c}_\alpha^{1/\alpha} \max\Big\{\Big(\frac{\sqrt{2}+1}{\sqrt{2}-1}\Big)^{1/\alpha},(\log n)^{1/\alpha}\Big(\frac{2}{\log 2}\Big)^{1/\alpha}\Big\},
    \]
i.\,e.\ the claim of Lemma \ref{maxOrl}, where we have set $C \coloneqq (\tilde{c}_\alpha c_\alpha^{-1})^{1/\alpha}$.
\end{proof}

\end{document}